\documentclass[reqno,english]{amsart}
\usepackage{etex}
\usepackage{amsmath,amssymb,amsthm,bbm,mathtools,comment}
\usepackage[shortlabels]{enumitem}
\usepackage[pdftex,colorlinks,backref=page,citecolor=blue]{hyperref}
\usepackage[mathscr]{euscript}
\usepackage{tikz,babel,adjustbox}
\usepackage{cmtiup}
\usepackage{amsfonts}
\usepackage{graphicx}
\usepackage{caption}
\usepackage{verbatim}
\usepackage{array}
\usepackage[frame,cmtip,arrow,matrix,line,graph,curve]{xy}
\usepackage{graphpap, color, pstricks}
\usepackage{pifont}
\usepackage{cmtiup}

\allowdisplaybreaks

\theoremstyle{plain}
\newtheorem{theorem}{Theorem}[section]

\newtheorem{claim}[theorem]{Claim}

\newtheorem{prop}[theorem]{Proposition}

\newtheorem{conjecture}[theorem]{Conjecture}

\theoremstyle{remark}

\newcommand{\RR}[0]{{\mathbb R}}

\newcommand{\E}[0]{\mathbb{E}}
\newcommand{\sub}[0]{\subseteq}
\newcommand{\sm}[0]{\setminus}

\newcommand{\beq}[1]{\begin{equation}\label{#1}}
\newcommand{\enq}[0]{\end{equation}}

\newcommand{\A}[0]{{\mathcal A}}
\newcommand{\B}[0]{{\mathcal B}}
\newcommand{\cee}[0]{{\mathcal C}}

\newcommand{\f}[0]{{\mathcal F}}

\newcommand{\g}[0]{{\mathcal G}}

\newcommand{\U}[0]{{\mathcal U}}

\newcommand{\ra}[0]{\rightarrow}
\newcommand{\Ra}[0]{\Rightarrow}

\newcommand{\0}[0]{\emptyset}

\newcommand{\C}[2]{\binom{#1}{#2}}
\newcommand{\Cc}[0]{\tbinom}
\newcommand{\ga}[0]{\alpha }
\newcommand{\gb}[0]{\beta }
\newcommand{\gc}[0]{\gamma }
\newcommand{\gd}[0]{\delta }

\newcommand{\gl}[0]{\lambda }

\newcommand{\gz}[0]{\zeta}

\newcommand{\supp}[0]{\mbox{\rm{supp}}}

\newcommand{\mn}[0]{\medskip\noindent}
\newcommand{\nin}[0]{\noindent}

\newcommand{\al}{\alpha}

\newcommand{\CSS}{\operatorname{C}^*}
\newcommand{\CSSS}{\operatorname{C}}

\newcommand\abs[1]{\left|#1\right|}

\newcommand{\tE}{\mathsf{E}}

\newcommand{\tw}{\mathsf{w}}

\newcommand{\ty}{\mathsf{y}}

\newcommand{\cB}{\mathcal{B} }
\newcommand{\cC}{\mathcal{C} }

\newcommand{\cE}{\mathcal{E} }

\newcommand{\cT}{\mathcal{T} }

\newcommand{\JJ}{R}

\title{On a problem of M.\ Talagrand}
\author{Keith Frankston}
\address{Center for Communications Research - Princeton, Princeton, NJ 08540, USA}
\email{k.frankston@fastmail.com}

\author{Jeff Kahn}
\address{Department of Mathematics, Rutgers University, Piscataway, NJ 08854, USA}
\email{jkahn@math.rutgers.edu}

\author{Jinyoung Park}
\thanks{The authors were supported by NSF grant DMS-1501962 and BSF Grant 2014290. The second author was also supported by NSF grant DMS-1954035, and the third directly by NSF
grant DMS-1926686
and indirectly by NSF grant CCF-1900460.}
\address{School of Mathematics, Institute for Advanced Study, Princeton, NJ 08540, USA}
\email{jpark@math.ias.edu}

\begin{document}

\maketitle
\begin{abstract}
We address a special case of a conjecture of M.\ Talagrand relating two
notions of ``threshold'' for an increasing family $\f$ of subsets of a finite set $V$.
The full conjecture implies equivalence of the ``Fractional Expectation-Threshold Conjecture,'' 
due to Talagrand and recently proved by the authors and B.\ Narayanan, and the (stronger) 
``Expectation-Threshold
Conjecture'' of the second author and G.\ Kalai.  
The conjecture under discussion here says there is a fixed $L$ such that if, for a given $\f$, 
$p\in [0,1]$ admits 
$\gl:2^V \ra \mathbb R^+$ with
\[
\mbox{$\sum_{S\sub F}\gl_S\geq 1 ~~\forall F\in \f$}
\]
and
\[
\mbox{$\sum_S\gl_Sp^{|S|} \leq 1/2$}
\]
(a.k.a.\ $\f$ is \emph{weakly p-small}),
then $p/L$ admits such a $\gl$ taking values in $\{0,1\}$
($\f$ is \emph{$(p/L)$-small}).  Talagrand showed this when $\gl$ 
is supported on singletons and suggested, as a more challenging test case, proving it when
$\gl$ is supported on pairs.  The present work provides such a proof.

\end{abstract}
\maketitle

\section{Introduction}\label{Intro}

Given a finite set $V$, write
$2^V$ for the power set of $V$
and, for $p\in [0,1]$, $\mu_p$ for the product
measure on $2^V$ given by $\mu_p(S) = p^{|S|}(1-p)^{|V\sm S|}$.
An $\f\sub 2^V$
is {\em increasing} 
if
$B\supseteq A\in\f\Ra B\in\f$.
For $\g\sub 2^V$
we use $\langle \g\rangle $
for the increasing family generated by $\g$, namely $ \{B\sub V:\exists A\in \g, B\supseteq A\}$.

We assume throughout that $\f\sub 2^V$ is increasing and not equal to $2^V,\0$.
Then $\mu_p(\f)(= \sum\{\mu_p(S): S \in \f\}$)
is strictly increasing in $p$,
and we define the {\em threshold}, $p_c(\f)$,
to be the unique $p$ for which
$\mu_p(\f)=1/2$.
(This is finer than the original Erd\H{o}s--R\'enyi notion,
according to which
$p^*=p^*(n)$ is \emph{\textbf{a}} threshold for $\f=\f_n$
if
$\mu_p(\f)\ra 0$ when $p\ll p^*$ and $\mu_p(\f)\ra 1$ when $p\gg p^*$.
That $p_c(\f)$ \emph{is} always an Erd\H{o}s--R\'enyi threshold 
follows from~\cite{BT}.)

Thresholds have been a---maybe \emph{the}---central concern of the study of random discrete structures 
(random graphs and hypergraphs, for example) since its initiation by Erd\H{o}s and R\'enyi~\cite{ER}, with much of that effort concerned with
identifying (Erd\H{o}s--R\'enyi) 
thresholds for specific properties (see~\cite{BBbook, JLR})---though it was 
not observed until~\cite{BT} that \emph{every} sequence of increasing properties
admits such a threshold. 

The main concern of this paper is the relation between the following two notions of
M.\ Talagrand~\cite{Talagrand4,Talagrand1,Talagrand}. 
(Our focus is Conjecture~\ref{littleTal} and our main result is Theorem~\ref{MT}; 
we will come to these following some motivation.)

Say $\f$ is
$p$-{\em small} if there is a $\g\sub 2^V$ such that
\beq{psmall1}
 \langle \g\rangle\supseteq \f
\enq
(that is, each member of $\f$
contains a member of $\g$)
and
\beq{psmall2}
\sum_{S\in \g}p^{|S|} \leq 1/2,
\enq
and set
$
q(\f)= \max\{\mbox{$p$ : $\f$ is $p$-small}\}.
$
Say $\f$ is
{\em weakly $p$-small} if there is a $\gl:2^V \ra \mathbb R^+$ ($:=[0,\infty)$)
such that
\beq{wpsmall1}
\sum_{S\sub F}\gl_S\geq 1 ~~\forall F\in \f
\enq
and
\beq{wpsmall2} 
\sum_S\gl_Sp^{|S|} \leq 1/2,
\enq 
and set
$
q_f(\f)= \max\{\mbox{$p$ : $\f$ is weakly $p$-small}\}.
$
As in~\cite{FKNP} we refer to $q(\f)$ and $q_f(\f)$ (respectively) as the \emph{expectation-threshold} and 
\emph{fractional expectation-threshold} of $\f$.  (Note the former is used slightly differently in ~\cite{KK}.)  
Notice that
\beq{q's}
q(\f)\leq q_f(\f)\leq p_c(\f).
\enq
(The first inequality is trivial and the second holds since, 
for $\gl$ as in \eqref{wpsmall1}, \eqref{wpsmall2} and $Y$
drawn from $\mu_p$,

\beq{triv.l.b.}
\mu_p(\f) \leq \sum_{F\in \f}\mu_p(F)\sum_{S\sub F}\gl_S  \leq
\sum_S\gl_S\mu_p(Y\supseteq S) =
\sum_S\gl_Sp^{|S|} \leq 1/2.) 
\enq

In particular, each of $q$, $q_f$ is a lower bound on $p_c$, and these turn out to be 
easily understood  (and to agree up to constant)
in many cases of interest; see~\cite{FKNP}.
The next two conjectures---respectively the main conjecture (Conjecture~1) of 
\cite{KK} and a sort of LP relaxation thereof suggested by 
Talagrand~\cite[Conjecture~8.3]{Talagrand}---say that these bounds are never far from the truth.
\begin{conjecture}\label{CKK}
There is a universal $K$ such that for every finite $V$ and increasing
$\f\sub 2^V$,
\[
p_c(\f) \le Kq(\f)\log |V|.
\]
\end{conjecture}

\begin{conjecture}\label{CT}
There is a universal $K$ such that for every finite $V$ and increasing
$\f\sub 2^V$,
\[   
p_c(\f) \le Kq_f(\f)\log |V|.
\]   
\end{conjecture}

\nin
Talagrand \cite[Conjecture~8.5]{Talagrand} 
also proposes the following strengthening of Conjecture~\ref{CT},
in which $\ell(\f)$ is the maximum size of a minimal member of $\f$.

\begin{conjecture}\label{CT'}
There is a universal $K > 0$ such that for every finite $V$ and increasing
$\f\sub 2^V$,
\[   
p_c(\f)< Kq_f(\f)\log \ell(\f).
\]   
\end{conjecture}
\nin
Conjecture~\ref{CT'} is proved in~\cite{FKNP}, to which we also refer for discussion of
the very strong consequences that originally motivated Conjecture~\ref{CKK},
but follow just as easily from Conjecture~\ref{CT}.

Turning, finally, to the business at hand, we are interested in the following conjecture
of Talagrand~\cite[Conjecture 6.3]{Talagrand}, which says that the parameters $q$ and
$q_f$ are in fact not very different.
\begin{conjecture}\label{littleTal}
There is a fixed L such that, for any $\f$,
$~q(\f)\geq q_f(\f)/L$.
\end{conjecture}
\nin
(That is, weakly $p$-small implies $(p/L)$-small.)
This of course implies  
\emph{equivalence of} Conjectures~\ref{CT} and~\ref{CKK}, as well as of 
Conjecture~\ref{CT'} and the corresponding strengthening of Conjecture~\ref{CKK};
in particular, in view of~\cite{FKNP}, Conjecture \ref{littleTal} 
would now supply a proof of Conjecture~\ref{CKK}.
(Post-\cite{FKNP} this implication is probably the best motivation for
Conjecture~\ref{littleTal}, but the authors have long been interested in the conjecture
for its own sake, as it would be a striking instance of a broad, natural class of
examples where the passage from an integer problem to its fractional counterpart
has only a minor effect on behavior.)

The following mild reformulation of Conjecture~\ref{littleTal} will be convenient.

\begin{conjecture}\label{talconj'}
There is a fixed $J$ such that for any $V,p \in [0,1]$ 
and $\gl:2^V \sm \{\0\}\rightarrow \mathbb R^+$, 
\beq{jp.target}
\{U\sub V:  \sum_{S\sub U}\gl_S\geq \sum_S\gl_S(Jp)^{|S|}\}
\enq
is $p$-small.
\end{conjecture}

As Talagrand observes, 
even simple instances of Conjecture~\ref{littleTal} 
are not easy to establish.
He suggests two test cases, which in the
formulation of Conjecture~\ref{talconj'} become:

\mn
(i)  $V=\C{[n]}{2}=E(K_n)$ and (for some $k$)
$\gl$ is the indicator of
$\{\mbox{copies of $K_k$ in $K_n$}\}$;

\mn
(ii)  $\gl$ is supported on 2-element sets.

\mn
(He does prove Conjecture~\ref{talconj'}
for $\gl$ supported on singletons; 
see Proposition~\ref{tfs} for a quantified version that will be useful in what follows.)

The quite specific (i) above was treated in~\cite{DK}.  
Here we dispose of the much broader (ii):
\begin{theorem}\label{MT}
Conjecture~\ref{talconj'} holds when $\supp(\gl)\sub \C{V}{2}$; 
in other words, there is a $J$ such that for any graph $G=(V,E)$, $p\in [0,1]$ and 
$\gl: E\ra \RR^+$, 
\beq{MTdisplay}
\{U\sub V:\gl(G[U])\geq J^2\gl(G)p^2\}
\enq
is $p$-small (where $G[U]$ is the subgraph induced by $U$).
\end{theorem}
\nin
(We could of course take $G=K_n$, but find thinking of a general $G$ more natural.)

It seems not impossible
that the ideas underlying Theorem~\ref{MT} can 
be extended to give Conjecture~\ref{littleTal} in full, but we don't yet see this.

The rest of the paper is devoted to the proof of Theorem~\ref{MT}.
The most important part of this turns out to be a version of the 
``unweighted'' case (that is, with $\gl$ taking values in $\{0,1\}$), though 
deriving Theorem~\ref{MT} from this still needs some ideas;
a precise statement (Theorem~\ref{Tr2}) is given 
in Section~\ref{Framework}, 
following a few preliminaries.
Section~\ref{Weighted} then proves Theorem~\ref{MT} assuming Theorem~\ref{Tr2},
and the proof of Theorem~\ref{Tr2} itself is given in 
Section~\ref{Unweighted}.

\section{Orientation}\label{Framework}

We use $[n] $ for $ \{1, 2, \ldots, n\}$, $2^X$ for the power set of $X$, 
and $\binom{X}{r}$ for the family of $r$-element subsets of $X$,
and recall from above that $\langle \A\rangle$ is the increasing family 
generated by $\A\sub 2^X$.
For a set $X$ and $p\in [0,1]$, $X_p$ is the ``$p$-random'' subset of $X$
in which each $x\in X$ appears with probability $p$ independent of other choices.
We assume throughout that $p$ has been specified and often omit it from our notation.

For $\A\sub 2^V$, the \emph{cost} of $\A$ (w.r.t.\ our given $p$) is 
$\CSSS(\A)=\sum_{S\in \A}p^{|S|}$.  
We say $\A$ \emph{covers} $\B\sub 2^V$ if 
$\langle \A\rangle \supseteq \B$, set 
\[
\CSS(\B) = \min \{\CSSS(\A): \mbox{$\A$ covers $\B$}\},
\]
and say $\B$ can be \emph{covered at cost $\gc$}
if $\CSS(\B)\leq \gc$.  
So ``$\cB$ $p$-small'' is the same as $\CSS(\cB) \le 1/2$,
and each of Conjecture~\ref{talconj'}, Theorem~\ref{MT} says (roughly) that the collection 
of subsets $U$ of $V$ for which $\sum_{S\sub U}\gl_S$ is much larger than the ``natural'' value,
$\E [\sum_{S\sub V_p} \gl_S]=\sum\gl_Sp^{|S|}$, admits such a ``cheap'' cover.
Talagrand's proof for singletons, to which we turn next, provides a first, simple 
illustration of this, and what we do in the rest of the paper \emph{amounts to} 
producing such a cover for the collection in \eqref{MTdisplay}.

\mn

\nin \textit{Singletons.}
In the above language, Talagrand's result for
$\gl$ supported on singletons becomes:

\begin{prop}\label{tfs}
For all $\gz:V\ra \RR^+$ and $J>2e$, 
\beq{tfscost}
\CSS(\{U\sub V: \gz(U)\geq J\gz(V)p\}) < 2e/(J-2e).
\enq
\end{prop}

\nin 
(The dependence on $J$ is best possible up to constants; e.g. take $|V|=J$, $p=J^{-2}$ 
and $\zeta \equiv 1$.  The switch from $\gl$ to $\gz$ will be convenient when we 
come to use the proposition; see \eqref{use2.1}.)

\begin{proof}
We may take $V=[n]$ and assume $\gz$ is non-increasing (and positive) and $Jp\le 1$ (since the 
statement is trivial when $Jp> 1$). Define $R$ by \[
\frac{1}{Rp}=\left\lceil\frac{1}{Jp}\right\rceil=: a.
\]

We claim that the collection
\[
\A=\bigcup_{k\ge 1} \binom{[ak]}{k}
\] covers the family in \eqref{tfscost}; this gives the proposition since the l.h.s. of \eqref{tfscost} is then at most \[
\CSSS(\A)=\sum_{k\ge1} \binom{ak}{k}p^k < \sum_{k\ge1} \left(\frac{e}{R}\right)^k< \frac{e}{R-e}< \frac{2e}{J-2e}
\] (the last inequality holding since $Jp\le 1$ implies $R>J/2$.)

To see that the claim holds, observe that its failure implies the existence of some $U=\{u_1<u_2<\cdots<u_\ell\}\subseteq [n]$ with $\gz(U)\ge J\gz(V)p$ such that $\abs{U\cap [ak]}< k$ for all $k>0$. But then $u_i>ia$ for all $i\in[\ell]$, yielding the contradiction \[\gz(V)>\sum_{i=0}^{\ell-1}\sum_{j\in [a]} \gz(j+ia) \ge a\gz(U)\ge \gz(V).\qedhere\]
\end{proof}

\bigskip

\nin \textit{Toward doubletons.} 
Graphs here are always simple and are mainly thought of as sets of edges;
thus $|G|$ is $|E(G)|$.
We use
$\nabla_G(v)$ or $\nabla_v$ for $\{e\in E(G):v\in e\}$; so the degree of 
$v$ is $d_v= |\nabla_v|$.
(We also use $N_G(v)$ for the neighborhood of $v$ in $G$.)

\nin 
The following convention will be helpful.
Given a graph $G$ on $V$, we associate with each $U\sub V$ a ``weighted subset'' 
$D(U)=D_G(U)$
of $E(G)$ that assigns to each $e$ the weight $|e\cap U|/2$.  
(We also use $D_v$ or $D_G(v)$ for $D(\{v\})$.)
We then have (or define), for any $\gl:G\ra\RR^+$,
\[
\mbox{$\gl(D(U)) =\frac{1}{2}\sum_{v\in U}\gl(\nabla_v)$}
\]
(e.g.\ $|D(U)| =\frac{1}{2}\sum_{v\in U}d_v$).  
Notice that 
\[
\E \gl(G[V_p]) = \E \gl(D(V_p))p
\]
(e.g.\ $\E |G[V_p]| = \E |D(V_p)|p$), so $\gl(D(U))p$ is a natural benchmark against which
to measure $\gl(G[U])$.

As mentioned at the end of Section~\ref{Intro}, the heart of our argument 
deals with the unweighted case of Theorem~\ref{MT}, where we are given some 
(simple) graph $G$ on $V$, and the collection in \eqref{MTdisplay} becomes the set of
$U$'s for which $G[U]$ is atypically large.  It is here that we are concerned with the
production of covers, which are then available for use in the weighted case.

For the derivation of Theorem~\ref{MT}, we will
decompose $G$ into subgraphs $G_1, G_2, \ldots$ so that the $\gl$-values of 
the edges within a $G_i$ are roughly equal, and show that for each ``heavy'' $U$
(meaning one with large $\gl(G[U])$, as in \eqref{MTdisplay}), there is some $i$ for 
which $G_i[U]$ is large.
We then plan to appeal to the unweighted case to cover,
for each $i$, the $U$'s that are ``heavy'' for $G_i$---a plan made delicate by the need to sum 
the contributions of many $G_i$'s to the l.h.s.\ of \eqref{psmall2}.

To deal with this we need a little more than the unweighted version of 
Theorem~\ref{MT}, as follows. 
Define 
\beq{CJ*}
\CSS_J(\mu,T)
\enq
to be the infimum of those $\gc$'s for which, for every
$p$ and (simple) graph $G$ (on $V$) with $|G|p^2\leq \mu$,
\beq{U}
\{U\sub V: |G[U]|\geq \max\{T, J|D_G(U)|p\}\}
\enq
can be covered at cost $\gc$.

The technical-looking requirement involving $D_G$ is a crucial feature of our argument:
for derivation of Theorem~\ref{MT}, we will
need cost bounds that improve as the ``target'' $T$ grows, even if $T/\mu$ does not,
which need not be the case without this extra condition
(e.g.\ it's not hard to see that if $G$ is the union of $(Kp)^{-1}$
disjoint copies of $K_{1,m}$,
and $T=mp=K\mu$, then, no matter how large $m$ is,
$\{U\sub V: |G[U]|\geq T\}$---or the smaller $\{U\sub V: G[U]\cong K_{1,mp}\}$---cannot 
be covered at cost less than $1/K$).
License to use 
the condition will be provided by the reduction 
to unweighted in Section~\ref{Weighted}.

Our central result is:

\begin{theorem}\label{Tr2}
For any $\mu$ and $T= cJ^2\mu$ with
\beq{Jc}
\mbox{$c\ge 256e/J$, $\,\,J\geq 8e$}
\enq
and $J_1=J/(8e)$,
\beq{Tr2cost}
\CSS_J(\mu,T) \leq 32 c^{-1}\min\{J_1^{-2},J_1^{-\sqrt{T}/16}\}.
\enq
\end{theorem}
\nin
(Here and throughout we don't worry about getting good constants,
trying instead to keep the argument fairly clean.)
As already mentioned, the proof of Theorem~\ref{Tr2} is given in Section~\ref{Unweighted},
following the derivation of Theorem~\ref{MT}, to which we now turn.

\section{proof of Theorem~\ref{MT}}\label{Weighted}

Here we assume Theorem~\ref{Tr2} and
prove the following quantified version of Theorem~\ref{MT}.
\begin{theorem}\label{MT'}
For any graph G on V, $\gl:G\ra \RR^+$ and 
\beq{JJlb}
\JJ\geq 4096\sqrt{2} e,
\enq
the set 
\[
\U_0=
\{U \sub V:\gl(G[U]) \ge \JJ^2\gl(G)p^2\}
\]
can be covered at cost $O(1/\JJ)$.
\end{theorem}

\nin
\emph{Proof.}
We take $G,\gl,\JJ$ to be as in the theorem, use $D(U)$ for $D_G(U)$
(defined in Section~\ref{Framework}), and assume throughout that 
\[
U\in \U_0.
\]

We first observe that it is enough to prove the theorem assuming
\beq{thetas}
\mbox{the only positive values taken by $\gl$ are
$\theta_i:=2^{-i}$, $~i =1,2,\ldots$} \hspace{10pt},
\enq
with \eqref{JJlb} slightly weakened to \beq{oldJJlb}
\JJ\ge 4096 e.
\enq
Then for a general $\gl$ (which we may of course scale to
take values in $[0,1]$) and $\gl'$ given by
\[
\gl'_S = \max\{\theta_i: \theta_i\leq \gl_S\},
\]
$\U_0$ as in the theorem is contained in the corresponding collection with
$\gl$ and $\JJ^2$ replaced by $\gl'$ and $\JJ^2/2$ (which supports \eqref{oldJJlb}),
since $U\in \U_0$ implies
$2\gl'(G[U])> \gl(G[U]) \geq \JJ^2\gl(G)p^2\geq \JJ^2\gl'(G)p^2$. So we assume from now on that $\lambda$ and $\JJ$ are as in \eqref{thetas} and \eqref{oldJJlb} (respectively).

Note also that
Proposition~\ref{tfs}, with $\gz(v)=\gl(D_v)$
(for which we have $\gz(V)=\sum\gz(v) =\frac{1}{2}\sum \gl(\nabla_v)=\gl(G)$
and $\gz(U) =\gl(D(U))$), says that the set
\beq{use2.1}
\{U\sub V: \gl(D(U))\geq \JJ\gl(G)p\}
\enq
admits a cover of cost less than $6/\JJ$.  So we specify such a 
cover as a first installment on $\g$ and it then becomes enough to 
show that 
\[
\U^*:=\{U\in \U_0: \gl(D(U))< \JJ\gl(G)p\}
\]
can be covered at cost $O(1/\JJ)$;  in fact we will show
\beq{C*U*}
\CSS(\U^*) = O(\JJ^{-2}).
\enq
We assume from now on that $U\in \U^*$.

\mn

Set $G_i =\{e\in G:\gl(e)=\theta_i\}$.
For \eqref{C*U*} we first show that, for each $U\in \U^*$, some
$G_i[U]$ must be ``large,'' meaning $U$ belongs to $\U_i$, defined in \eqref{jp.heavy},
and then bound the costs of the $\U_i$'s using Theorem~\ref{Tr2}.

From this point we use $D_i(U)$ for $D_{G_i}(U)$.
We observe that for any $H \sub G$,
\[
\gl(H)=\sum_i \theta_i |H\cap G_i|
,
\]
and abbreviate
\[
\tw_i=\gl(G_i)=\theta_i|G_i|, \quad \tw=\gl(G)=\sum \tw_i.
\]

Given $U$, define $L=L(U)$, $K=K(U)$, $L_i=L_i(U)$ and $K_i=K_i(U)$ by
\begin{align}
\gl(D(U))&=L\tw p,\nonumber\\
\gl(G[U])&=KL\tw p^2,\nonumber\\
|D_i (U)|&=L_i|G_i|p,\label{Li}
\end{align}
and
\beq{Ki}
|G_i[U]|=K_iL_i|G_i|p^2.
\enq
Then
\beq{LiL}
L\tw p =    
\sum \theta_i|D_i(U)|=
\sum L_i \tw_ip
\enq
and
\[
KL\tw p^2 = 
\sum \theta_i |G_i[U]|=
\sum K_iL_i\tw_i p^2 .
\]
Since $U\in \U_0$, we have 
\beq{sumKL}
\sum K_iL_i\tw_i \ge \JJ^2\tw ,
\enq
while $U\in \U^*$ gives 
\beq{LJ}
L< \JJ.
\enq
Note also that, with
\[
I=I(U)=\{i:K_i>\JJ/2\},
\]
we have
\beq{sumKLw}
\sum \{K_iL_i\tw_i:i \in I\} >\JJ^2\tw /2,
\enq
as follows from \eqref{sumKL}
and (using \eqref{LiL} and \eqref{LJ})
\[
\sum\{K_iL_i\tw_i:i\not\in I\} \le (\JJ/2) L\tw < \JJ^2\tw/2 .
\]

\mn

Now let $\tE_i=|G_i|p^2$ ($=\E |G_i[V_p]|$) and, for integer $\ga$,
\[
\cE_{\al}=\{i :\tE_i \in (2^{\al-1},2^{\al}]\}.
\]
We arrange the $i$'s in an array, with columns indexed by $\al$'s (in increasing order) and column $\al$ consisting of the indices in $\cE_\ga$, again in increasing order. 
(So $\tw_i$'s within a column 
\emph{decrease} as we go down.  Note column lengths may vary.) Define $\cB_\beta$ to be the set of indices in row $\beta$.

\mn

\begin{table}[h!]
\centering

\begin{tabular}{|c||c c| c| c c }
 \hline
  & $\cdots$ & $\al-1$ & $\al$ & $\al+1$ & $\cdots$  \\ 
 \hline\hline
  1&&&\\ 
 \vdots&&&  \\
 \hline
 $\beta$&&& $i$  \\
 \hline
 \vdots&&&
\end{tabular}
\\
\vspace{.5cm}
\caption{
$i$ is the $\beta$th smallest index in $\cE_\al$ (when $\abs{\cE_\al}\ge \beta$).}
\end{table} 
Set $\ty_i=\theta_i 2^\al/p^2$ (for $i \in \cE_\ga$) and $\ty=\sum_{i\ge 1}\ty_i$, noting that
\[
\ty_i/2 < \tw_i \le \ty_i.
\]
Set 
\[
c^*_\beta=(3/2)^{\beta-1}\JJ^2/16 ~~~(\beta \ge 1)
\]
and $c_i=c_\beta^*$ if $i\in\cB_\beta$.
Let $\tw^*_\beta$ and $\ty^*_\beta$ be (respectively) the sums of the 
$\tw_i$'s and $\ty_i$'s over $i\in \cB_\gb$, and notice that
\[\ty^*_{\beta+1} \le \ty^*_{\beta}/2 \quad \mbox{ for $\beta \ge 1$}\]
(since $i=\cB_{\gb+1}\cap\cE_\ga$---where we abusively use $i$ for $\{i\}$---implies $i>j:=\cB_\gb\cap\cE_\ga$, whence $2\ty_i\le \ty_j$).

\begin{claim}\label{obs:Gi} For each $U\in \U^*$ there is an $i \in I(U)$ with $K_i(U)L_i(U)>c_i$.
\end{claim}

\begin{proof} With $\sum^\star$ denoting summation over $I(U)$, we have
(using \eqref{sumKLw} at the end)  
\[\begin{split} {\sum}^\star c_i \tw_i\leq \sum c_\beta^* \tw_\beta^* & \le \sum c^*_\beta \ty^*_\beta\\
&\le \ty_1^*(c_1^*+c_2^*/2+c^*_3/2^2 + \cdots)\\
&\le \ty(c_1^*+c_2^*/2+c^*_3/2^2 + \cdots)\\
&\le (\JJ^2/4) \ty <(\JJ^2/2)\tw <{\sum}^\star K_i(U) L_i(U) \tw_i.
\end{split}\]
\end{proof}

It follows that if, for each $i$, $\g_i$ covers 
\beq{jp.heavy}
\U_i:= \{U\sub V:  i\in I(U); ~K_i(U)L_i(U) >c_i\},
\enq
then $\cup\g_i$ covers $\U^*$; 
so we have
\beq{C*sum}
\CSS(\U^*)\leq \sum_i\CSS(\U_i).
\enq
On the other hand,
if $(\ga,\gb)$ is the pair corresponding to $i$ 
(that is, $i$ is the $\gb$th entry in column $\ga$ of our array), then 
(see \eqref{CJ*}, \eqref{U} for $\CSS_J$)
\[\CSS(\U_i)\leq \CSS_{\JJ/2}(2^\ga,T_{\ga,\gb}),\]
where $T_{\ga,\gb} = \max\{c_\gb^* 2^{\ga-1},1\}$
(since $|G_i|p^2=\tE_i\leq 2^\ga$, while $U\in \U_i$ implies,
using \eqref{Li}, \eqref{Ki} and $i\in I(U)$,
\[
|G_i[U]|=K_i(U)L_i(U)|G_i|p^2 \left\{\begin{array}{l}
> c_i|G_i|p^2 > c_\gb^*2^{\ga-1}\\
= K_i|D_i(U)|p > (\JJ/2)|D_i(U)|p).
\end{array}\right.
\]
So, with $\ga$ and $\gb$ ranging over integers and positive integers respectively,
\eqref{C*U*} will follow from
\beq{jp.goal} \sum
\CSS_{R/2} (2^\alpha, T_{\alpha, \beta}) =O(R^{-2}). \enq

\begin{proof}[Proof of \eqref{jp.goal}]

For $T_{\ga,\gb}= 1$ we bound $\CSS_{R/2}(2^\alpha, T_{\alpha, \beta})$
by $2^\ga$, using the trivial
\beq{jp.mu1}
\CSS_J(\mu,1)\leq \mu
\enq
(since $\{\{x,y\}:xy \in G\}$ itself covers the set in \eqref{U}), which---since 
$T_{\ga,\gb}= 1$ iff $2^\ga \leq 32 \JJ^{-2} (2/3)^{\gb-1}$---bounds the contribution of such pairs 
to the sum in \eqref{jp.goal} by 
\beq{smallTi}
\sum_{\beta}\sum_{\al:T_{\ga,\gb}=1}2^\al\le
64 \JJ^{-2}\sum_{\beta} (2/3)^{\gb-1} = 3\cdot 64 \JJ^{-2}.
\enq

For $T_{\alpha, \beta} >1$ we use Theorem~\ref{Tr2} with $T=T_{\ga,\gb}$ ($=c_\gb^*2^{\ga-1})$,
$\mu=2^\ga$, $J=\JJ/2$, and (thus)
\[
c = T/(\mu J^2)=c_\gb^*/(2J^2) = (3/2)^{\gb-1}/8.
\]
Note that \eqref{oldJJlb} gives $J\ge 8e$ and $c\ge 256e/J$, so \eqref{Jc} holds.

(Here, finally, we see the role of $C_J^*$ mentioned in the paragraph following \eqref{U}:
for a given $\gb$ we may be summing over many $\ga$'s with the same $T/\mu$, so
it's crucial that we have cost bounds that shrink with $T$ when this ratio
is held constant.)

For each integer $s\ge 0$ let $\cT_s=\{(\ga,\gb) : T_{\ga,\gb}\in(2^s,2^{s+1}]\}$.
For each $\beta\ge 1$ there is a unique $\al$ such that $(\ga,\gb)\in \cT_s$, and every $(\ga,\gb)$ with $T_{\ga,\gb}>1$ is in some $\cT_s$. Let $f(s)=\min\{J_1^{-2},J_1^{-2^{s/2-4}}\}$. Then for fixed $s$, we have (see \eqref{Tr2cost})
\beq{finb}
\sum_{(\ga,\gb)\in\cT_s}\CSS_{J}(2^{\ga},T_{\ga,\gb})\le\sum_{\beta} 32c^{-1}f(s) =\sum_{\beta}256\left(\frac{2}{3}\right)^{\gb-1}f(s)<3\cdot 256f(s),
\enq
and summing over all $s$ we get
\beq{largeTab}
\sum_{T_{\ga,\gb}>1} \CSS_J(2^{\ga},T_{\ga,\gb})< 
\sum_{s\ge 0} 768 f(s)
= \sum_{s\ge 0} 768\min\{J_1^{-2},J_1^{-2^{s/2-4}}\}=O(J_1^{-2}).
\enq
Finally, combining \eqref{largeTab} and \eqref{smallTi} gives \eqref{jp.goal}.
\end{proof}

\section{Proof of Theorem~\ref{Tr2}} \label{Unweighted}

Aiming for simplicity, we just bound the cost in \eqref{Tr2cost} assuming
\[
T=2^{2k+3}
\]
for some positive integer $k$ and
\beq{oldJc}
c=T/(\mu J^2)\ge 64e/J,
\enq
showing that in this case
\beq{specialT}
\CSS_J(\mu,T) \leq 8 c^{-1}
J_1^{-2^{k-1}-1}.
\enq

Before proving this, we show it implies Theorem~\ref{Tr2}, which, since
$\CSS_J(\mu,t)$ is decreasing in $t$, just requires showing that the
r.h.s.\ of \eqref{Tr2cost} bounds $\CSS_J(\mu,T_0)$ for some $T_0\leq T$.

If $T< 32$ this follows from the trivial  \eqref{jp.mu1},
since $\mu =T/(cJ^2)< 32c^{-1}J_1^{-2},$ matching the bound in \eqref{Tr2cost}.
Suppose then that $T\geq 32$ and let $T_0 = c_0J^2\mu $ be the largest integer
not greater than $T$
of the form $2^{2k+3}$ (with positive integer $k$).  We then have $c_0>c/4$ (supporting \eqref{oldJc}) 
and $2^{k-1} > \sqrt{T_0}/8> \sqrt{T}/16$, and it follows that the bound
on $\CSS_J(\mu,T_0)$ given by \eqref{specialT} is less than the bound in \eqref{Tr2cost}.

\mn
\emph{Proof of \eqref{specialT}.}
We have $|G|p^2 \le \mu$, $T=2^{2k+3}$ ($=cJ^2\mu$ with $J$ as in \eqref{Jc} and $c$ as in \eqref{oldJc}), and, with
\beq{U'}
\U:= \{U\sub V: |G[U]| > \max\{T, J|D_G(U)|p\}\},
\enq
want to show that $\CSS(\U)$ is no more than the bound in \eqref{specialT}.

Here, finally, we come to specification of a cover, $\g$.
Each member of $\g$ will be a disjoint union of stars (a.k.a.\ a \emph{star forest}),
where for present purposes a \emph{star at v in W} ($\sub V$) is some $\{v\}\cup S\sub W$ with 
$S\sub N_{G}(v)$.  (Where convenient we will also refer to this as the ``star $(v,S)$.'')
We say such a star is \emph{good} if
\beq{Jcond}
|S|\geq  Jd_v p/4.
\enq
Given a positive integer $L$, we define \beq{Lv}
L^v= \max\{L,\lceil Jd_v p/4 \rceil \}
\enq
and say a star $(v,\,S)$ is \emph{$L$-special} if $|S|=L^v$.

For positive integers $b$ and $L$, let $\g(b,L)$ ($\sub 2^V$) 
consist of all disjoint unions of $b$ $L$-special stars in $G$. 
We will specify a particular collection $\cC$ of pairs $(b,L)$ 
and set 
\[
\g=\cup\{\g(b,L):(b,L)\in \cC\}.
\]
Theorem~\ref{Tr2} is then given by the following two assertions.
\begin{claim}\label{Claim1}
$\g$ covers $\U$.
\end{claim}
\begin{claim}\label{Claim2}
$\CSSS(\g)$ is at most the bound in \eqref{specialT}.
\end{claim}
\nin

Set (with $i\in [k]$ throughout) $L_i=2^{i-1}$ and
\beq{gdi}
\gd_i = \max\{2^{-(i+2)},2^{i-k-3}\} 
\ge 1/(8L_i),
\enq
and notice that
\beq{gd12}
\sum\gd_i\le \sum2^{-(i+2)}+\sum 2^{i-k-3}\leq 1/2.
\enq
Let
\beq{bi}
b_i=\delta_i  4^{-i}T
\ge 2^{k-i}.
\enq
Finally, set 
\[
\cee =\{(b_i,L_i):i\in [k]\}.
\]

\begin{proof}[Proof of Claim~\ref{Claim1}]
We are given $U\in \U$ and must show it contains a member of $\g$.
Let $U_0=U$ and for $j=1,\ldots$ until no longer possible do:
let $(v_j,S_j)$, with $S_j=N_G(v_j)\cap U_{j-1}$, be a largest good star in $U_{j-1}$, 
and set $d_j=|S_j|$ and $U_j=U_{j-1}\sm (\{v_j\}\cup S_j)$.

The passage from $U_{j-1}$ to $U_j$ deletes at most $d_j^2$ edges that contain vertices
of $S_j$ of $U_{j-1}$-degree at most $d_j$; 
any other edge deleted in this step contains $u\in S_j$ with $U_{j-1}$-degree less than
$Jd_up/4$ (or $u$, having $U_{j-1}$-degree greater than $d_j$, would have been 
chosen in place of $v_j$); 
and of course each vertex $u$ of  the final $U_j$
has $U_j$-degree less than $Jd_up/4$.  
We thus have

\[
 |G[U]| \leq \sum_j d_j^2 +\sum_{v\in U} Jd_v p/4 \leq \sum_j d_j^2 + |G[U]|/2
\] 
(using the second bound in \eqref{U'}), so
\beq{sumdj}
\sum_j d_j^2 \ge |G[U]|/2 ~\geq T/2.
\enq

Set
\[
B_i=\left\{\begin{array}{ll}
\{j : d_j\in [2^{i-1},2^i)\}&\mbox{if $i\in [k-1]$,}\\
\{j : d_j\geq 2^{k-1}\}&\mbox{if $i=k$.}
\end{array}\right.
\]
(It may be worth noting that, while the $d_j$'s are decreasing, the degrees corresponding to $B_i$ 
increase with $i$.) In view of \eqref{sumdj}, either $\abs{B_k}\ge 1$ or (using \eqref{gd12})
\[
\sum_{i\in [k-1]} |B_i| 4^i \geq T/2\ge \sum_{i\in [k-1]} \delta_i T =\sum_{i \in [k-1]} b_i4^i.
\] 
Since $b_k=1$, it follows
that for some $i\in [k]$ we have
\beq{ineq:b}
|B_i| \geq  b_i.
\enq
On the other hand, since $j\in B_i$ implies 
$\abs{S_j}\ge L_i^v$ ($=\max\{L_i, \lceil Jd_vp/4 \rceil\})$,
the set $\bigcup \{S_j\cup \{v_j\} : j\in B_i\}$ contains some 
$W\in \g(b_i,L_i)(\subseteq \g)$ whenever $i$ is as in \eqref{ineq:b}.
This completes the proof of Claim~\ref{Claim1}.
\end{proof}

\begin{proof}[Proof of Claim~\ref{Claim2}.]
We first bound the costs, say $\CSSS(b,L)$, 
of the collections $\g(b,L)$.
Given $(b,L)$, set 
\[
q_v = p\left(\frac{ed_v p}{L^v}\right)^{L^v}.
\]
Then  $q_v$ bounds the total cost of the set of 
$L$-special stars at $v$ (using $\C{d_v}{L^v}\leq (ed_v/L^v)^{L^v}$), 
and it follows that 
\beq{CSS}
\CSSS(b,L)\leq \sum \Big\{\prod_{v\in B} q_v : B\in \Cc{V}{b}\Big\}.
\enq
For a given value of
$\varphi:=\sum_{v\in V} q_v$, the r.h.s.\ of \eqref{CSS} is largest when the $q_v$'s 
are all equal (this just uses $xy\leq [(x+y)/2]^2$), whence
 \begin{equation}\label{ineq:CS}
    \CSSS(b,L)\le \binom{\abs{V}}{b}\left(\frac{\varphi}{\abs{V}}\right)^b\le \left(\frac{e\varphi}{b}\right)^b.
\end{equation}
Recalling \eqref{Lv}, we have
\[q_v \le d_v p^2 \cdot \frac{e}{L}\left(\frac{4e}{J}\right)^{L-1},\] 
so (since $|G|p^2 \le \mu$)
\begin{equation}\label{phi}
\varphi\le  2\mu \cdot \frac{e}{L}\left(\frac{4e}{J}\right)^{L-1}.
\end{equation}

Now using \eqref{ineq:CS} and \eqref{phi}, recalling that $T=cJ^2\mu$,
$L_i=2^{i-1}$, 
$b_i=\gd_i 4^{-i}T= \gd_i T/(4L_i^2)$ and $J_1 = J/(8e)$, and for the moment
omitting the subscript $i$, we have
(with the final inequality \eqref{Lb} justified below)
\begin{align}
\CSSS(b,L)
    &\le \left[\frac{2e^2\mu}{L}\frac{4L^2}{\gd T}\left(\frac{4e}{J}\right)^{L-1}\right]^b\nonumber\\
    &=\left[8e^2L\cdot\frac{1}{c  J^2\gd}\left(\frac{4e}{J}\right)^{L-1}\right]^b\nonumber\\
    &=\left[c^{-1}\frac{L}{2\gd}\left(\frac{4e}{J}\right)^{L+1}\right]^b\nonumber\\
    &\leq \left[\frac{c}{4}\cdot J_1^{L+1}\right]^{-b}.\label{Lb}
\end{align}
For \eqref{Lb}, or the equivalent
\beq{finalL}
2^{L+4}\gd \geq L,
\enq
it is enough to show $2^{L+1}\geq L^2$ (since $\delta\ge 1/(8L)$; see \eqref{gdi}), which is true for positive integer $L$. 

\mn

Finally, returning to Claim~\ref{Claim2} (and recalling that $L$ and $b$ in the display ending
with \eqref{Lb} are really $L_i$ and $b_i$), we have 
\beq{sumCbiLi}
\CSSS(\g) =\sum_{i=1}^k\CSSS(b_i,L_i) \leq 
\sum_{i=1}^k \left[\frac{c}{4}\cdot J_1^{L_i+1}\right]^{-b_{i}}.
\enq
We use $b_i\ge 2^{k-i}$ (see \eqref{bi}) and $L_i=2^{i-1}$
to bound the r.h.s.\ of \eqref{sumCbiLi} by 
\beq{CG}\sum_{i=1}^k \left[\frac{cJ_1^{2^{i-1}+1}}{4}\right]^{-2^{k-i}} = \sum_{i=1}^k J_1^{-2^{k-1}}\left[\frac{cJ_1}{4}\right]^{-2^{k-i}} = \sum_{j=0}^{k-1} \left(\frac{c}{4}J_1^{2^{k-1}+1}\right)^{-1}\left[\frac{cJ_1}{4}\right]^{1-2^{j}};\enq
but, since $cJ_1/4 \ge 2$ (using \eqref{oldJc} and $J_1=J/(8e)$), 
the last expression in \eqref{CG} is less than the bound $8c^{-1} J_1^{-2^{k-1}-1}$ in \eqref{specialT}.
\end{proof}

\end{document}